\documentclass[reqno,11pt]{amsart}
\usepackage{amssymb,enumerate}
\usepackage{amsmath}
\usepackage{mathrsfs}
\usepackage{hyperref}
\usepackage{dsfont}
\usepackage{enumitem}

\usepackage{snapshot}

\allowdisplaybreaks

\begin{document}

\renewcommand{\AA}{\mathcal{A}}
\newcommand{\BB}{\mathcal{B}}
\newcommand{\CC}{\mathcal{C}}
\newcommand{\DD}{\mathcal{D}}
\newcommand{\EE}{\mathcal{E}}
\newcommand{\FF}{\mathcal{F}}
\newcommand{\GG}{\mathcal{G}}
\newcommand{\HH}{\mathcal{H}}
\newcommand{\II}{ I\hspace{-0.5mm}I}
\newcommand{\JJ}{\mathcal{J}}
\newcommand{\KK}{\mathcal{K}}
\newcommand{\LL}{\mathcal{L}}
\newcommand{\MM}{\mathcal{M}}
\newcommand{\NN}{\mathcal{N}}
\newcommand{\OO}{\mathcal{O}}
\newcommand{\PP}{\mathcal{P}}
\newcommand{\QQ}{\mathcal{Q}}
\newcommand{\RR}{\mathcal{R}}
\renewcommand{\SS}{\mathcal{S}}
\newcommand{\TT}{\mathcal{T}}
\newcommand{\UU}{\mathcal{U}}
\newcommand{\VV}{\mathcal{V}}
\newcommand{\WW}{\mathcal{W}}
\newcommand{\XX}{\mathcal{X}}
\newcommand{\YY}{\mathcal{Y}}
\newcommand{\ZZ}{\mathcal{Z}}


\newcommand{\A}{\mathbb{A}}
\newcommand{\B}{\mathbb{B}}
\newcommand{\C}{\mathbb{C}}
\newcommand{\D}{\mathbb{D}}
\newcommand{\E}{\mathbb{E}}
\newcommand{\F}{\mathbb{F}}
\newcommand{\G}{\mathbb{G}}
\renewcommand{\H}{\mathbb{H}}
\newcommand{\I}{\mathbb{I}}
\newcommand{\J}{\mathbb{J}}
\newcommand{\K}{\mathbb{K}}
\renewcommand{\L}{\mathbb{L}}
\newcommand{\M}{\mathbb{M}}
\newcommand{\N}{\mathbb{N}}
\renewcommand{\O}{\mathbb{O}}
\renewcommand{\P}{\mathbb{P}}
\newcommand{\Q}{\mathbb{Q}}
\newcommand{\R}{\mathbb{R}}
\renewcommand{\S}{\mathbb{S}}
\newcommand{\T}{\mathbb{T}}
\newcommand{\U}{\mathbb{U}}
\newcommand{\V}{\mathbb{V}}
\newcommand{\W}{\mathbb{W}}
\newcommand{\X}{\mathbb{X}}
\newcommand{\Y}{\mathbb{Y}}
\newcommand{\Z}{\mathbb{Z}}


\newcommand{\al}{\alpha}
\newcommand{\be}{\beta}
\newcommand{\ga}{\gamma}
\newcommand{\de}{\delta}
\newcommand{\ep}{\varepsilon}
\newcommand{\ze}{\zeta}
\newcommand{\et}{\eta}
\newcommand{\vth}{\vartheta}
\renewcommand{\th}{\theta}
\newcommand{\io}{\iota}
\newcommand{\ka}{\kappa}
\newcommand{\la}{\lambda}
\newcommand{\rh}{\rho}
\newcommand{\si}{\sigma}
\newcommand{\ta}{\tau}
\newcommand{\up}{\upsilon}
\newcommand{\ph}{\varphi}
\newcommand{\ch}{\chi}
\newcommand{\ps}{\psi}
\newcommand{\om}{\omega}

\newcommand{\Ga}{\Gamma}
\newcommand{\De}{\Delta}
\newcommand{\Th}{\Theta}
\newcommand{\La}{\Lambda}
\newcommand{\Si}{\Sigma}
\newcommand{\Up}{\Upsilon}
\newcommand{\Ph}{\xi}
\newcommand{\Om}{\Omega}

\newcommand{\inj}{\hookrightarrow}
\newcommand{\stetein}{\overset{s}{\hookrightarrow}}
\newcommand{\dichein}{\overset{d}{\hookrightarrow}}
\newcommand{\pa}{\partial}
\newcommand{\re}{\restriction}
\newcommand{\tief}{\downharpoonright}

\newcommand{\bra}{\langle}
\newcommand{\ket}{\rangle}
\newcommand{\bs}{\backslash}
\newcommand{\divv}{\operatorname{div}}
\newcommand{\Dt}{\frac{\mathrm d}{\mathrm dt}}

\newcommand{\sm}{\setminus}
\newcommand{\es}{\emptyset}

\newtheorem{theorem}{Theorem}
\newtheorem{corollary}{Corollary}
\newtheorem*{main}{Main Theorem}
\newtheorem{lemma}[theorem]{Lemma}
\newtheorem{proposition}{Proposition}
\newtheorem{conjecture}{Conjecture}
\newtheorem*{problem}{Problem}
\theoremstyle{definition}
\newtheorem{definition}[theorem]{Definition}
\newtheorem{remark}{Remark}
\newtheorem*{notation}{Notation}

\newcommand{\cqfd}{\begin{flushright}\vspace*{-3mm}$\Box $\vspace{-2mm}\end{flushright}}
\newcommand{\saut}{\vspace*{1mm}\\ \nd }
\newcommand{\on}{\mbox{ on }}
\newcommand{\with}{\mbox{ with }}
\newcommand{\nd}{\noindent}
\newcommand{\eps}{\varepsilon}

\sloppy
\title[]
{Global existence for diffusion-electromigration systems in space dimension three and higher}
 \author[D. Bothe]{Dieter Bothe}
 \address{Center of Smart Interfaces, Technische Universit\"at Darmstadt, Alarich-Weiss-Str. 10, 64287 Darmstadt, Germany}
 \email{bothe@csi.tu-darmstadt.de}
 \author[A. Fischer]{Andr\'e Fischer}
 \address{Center of Smart Interfaces, Technische Universit\"at Darmstadt, Alarich-Weiss-Str. 10, 64287 Darmstadt, Germany}
 \email{fischer@csi.tu-darmstadt.de}
 \author[M. Pierre]{Michel Pierre}
 \address{Ecole Normale Sup\'erieure de Rennes, IRMAR, UEB, Campus de Ker Lann, 35170 Bruz, France}
 \email{michel.pierre@ens-rennes.fr}

 \author[G. Rolland]{Guillaume Rolland}
 \address{Ecole Normale Sup\'erieure de Rennes, IRMAR, UEB, Campus de Ker Lann, 35170 Bruz, France}
 \email{guillaume.rolland@ens-rennes.fr}

\date{\today}
\thispagestyle{empty}
%
\maketitle
\begin{abstract}
We prove existence of global weak solutions for the Nernst-Planck-Poisson problem which describes the evolution of concentrations of charged species $X_1,\ldots,X_P$ subject to Fickian diffusion and chemical reactions in the presence of an electrical field, including in particular the Boltzmann statistics case. In contrast to the existing literature, existence is proved in any dimension. Moreover, we do not need the assumption $P=2$ nor the assumption of equal diffusivities for all $P$ components. Our approach relies on the intrinsic energy structure and on an adequate nonlinear and curiously more regular approximate problem.  The delicate passing to the limit is done in adequate functional spaces which lead to only weak solutions.
\end{abstract}






\section{Introduction and Main Results}
Our main goal is to prove global existence of weak solutions in any dimension for the Nernst-Planck-Poisson system (NPP). It describes the evolution of a dilute solution with charged solutes in presence of Fickian diffusion and electromigration. It consists in a coupled system of parabolic and elliptic equations for the unknowns $c_i$, $\Phi$, where $c_i$ denotes the concentration of species $X_i$ and $\Phi$ the electrical potential. In this model, the total mass flux of $X_i$ is given by
\[J_i=j_i^d+j_i^e=-d_i\nabla c_i-d_i\tfrac{F}{RT}z_ic_i\nabla\Phi,\]
where $j_i^d$ and $j_i^e$ represent the diffusional and the migrational part of the flux, respectively. The parameters $F,R,T>0$ denote the Faraday constant, the ideal gas constant and the (constant) temperature, respectively, and $z_i\in\Z$ represents the charge number of species $X_i$. In general, the diffusivities $d_i>0$ depend on the full composition of the system, see e.g.~\cite{cussler}.

The bulk equations for concentrations $c_i$ are given by
\begin{equation}\label{NPP1}
 \partial_tc_i +\textrm{div}(-d_i\nabla c_i - d_i\frac{F}{RT}z_i c_i\nabla \Phi)=f_i(c),\qquad\ t>0,\ x\in\Omega,\quad i=1,\ldots,P,
\end{equation}
where $\Omega\subset\R^N$, supplemented with no-flux boundary conditions and initial conditions, i.e.
\begin{equation}\label{NPP2}
\pa_\nu c_i+\frac{F}{RT}z_ic_i\pa_\nu\Phi=0,\quad t>0,\ x\in\pa\Omega,\qquad c_i(0)=c_i^0,\quad x\in\Omega.
\end{equation}
The right-hand side $f_i(c)$ is assumed to be quasi-positive, i.e.\ $f_i(c)\ge0$ if $c_i=0$, which allows us to expect nonnegative solutions $c_i$, if $c_i^0$ is nonnegative.

The electrical potential $\Phi$ is determined by Maxwell's equation of electro-statics, i.e.
\begin{equation}\label{NPP3}
-\epsilon\Delta \Phi=F\sum_{i=1}^P z_ic_i,\qquad\ t>0,\ x\in\Omega,
\end{equation}
where we assume constant permittivity $\epsilon>0$ of the fluid. Note that the right-hand side represents the charge density within the electrolyte. Coming up with physically reasonable boundary conditions for this problem is a more delicate topic. It seems well-accepted in the mathematical literature to impose homogeneous Dirichlet or Neumann boundary conditions although, as is well-known, the boundary is, in general, charged (see~\cite{Newman}) which is not accounted for in those frameworks. This is the reason why we work with the inhomogeneous Robin boundary condition
\begin{equation}\label{NPP4}
\pa_\nu\Phi+\tau\Phi=\xi,\qquad t>0,\ x\in\pa\Omega.
\end{equation}
This condition can be motivated by considering the boundary locally as a plate capacitor. The parameter $\tau>0$ can be viewed as the capacity of the boundary and $\xi$ (a given datum to the problem) refers to an external potential as well as a boundary charge. For a more detailed discussion in this respect, we refer to~\cite{BFS12}.  Actually, we will assume that $\tau=\tau(x)\not \equiv 0$ so that pure Neumann boundary conditions may be considered on some parts of the boundary.

The NPP-model goes back to the fundamental works of W.~Nernst and M.~Planck, see~\cite{Ne89,Pl90a,Pl90b}. Typical situations which are captured by this model comprise semiconductors, electrolytes, nano-filtration processes, ion channels, etc., see e.g.~\cite{Selberherr,Newman,BothePruss,SNE}. A quite recent discussion on the applicability of NPP in the general case of non-dilute aqueous solutions is contained in~\cite{dreyer12}. There, the authors discuss the necessity of incorporating the mechanical pressure into the model. Indeed, for the case of a non-dilute fluid more general models are required in order to assure thermodynamic consistency, see also~\cite{deGrootMazur-book}. For a justification of NPP in the dilute case, we refer to~\cite{BFS12} where a detailed account on the connection of NPP to the more general Maxwell-Stefan equations and to the frequently employed ``electro-neutrality condition'' is also given.

Throughout this work, we assume that the diffusivities $d_i$ are functions of space and time only subject to some further mild regularity constraints, where we allow $d_i\ne d_j$ for $i\ne j$. In view of the production terms $f_i$, we restrict ourselves to the case of bounded functions which depend on $t,x,c$; see Remark~\ref{rem:data} for further explanation and motivation. 

{\em The purpose of the present article is to show existence of global weak solutions} to system~(\ref{NPP1})-(\ref{NPP4}) for bounded domains $\Omega\subset\R^N$ with $C^2$-boundary, where $N\ge1$. In general, the concentrations $c_i$ turn out to only have local $W^{1,1}$-regularity in space. For the physical case $N=3$, however, we are able to reveal global $W^{1,\frac32}$-regularity.
We do not address the uniqueness question which is an open problem in this weak setting.

In the 1990s, the Nernst-Planck-Poisson system has attracted quite some attention in mathematical research; see, e.g., \cite{Biler92,BD00,BHN94,CL92,CL_multidim,gaj85,GG96,GGH95,Glitzky,Glitzky_2}. Let us give a brief account on those results and show their relation to ours.

In most cases, a bounded smooth domain $\Omega\subset\R^N$, $N=2,3$, is considered. In case of $N=2$ well-posedness and long-time behaviour are well-understood: in \cite{BHN94}, existence and uniqueness of global weak solutions is shown, as well as convergence to uniquely determined steady states. For sufficiently smooth data, it is proved in \cite{CL_multidim} that there is a unique global classical solution. These results have been improved in \cite{BD00} by computing an explicit exponential rate of convergence with the help of logarithmic Sobolev inequalities. In papers of H.~Gajewski, A.~Glitzky, K.~Gr\"oger, R.~H\"unlich, \cite{gaj85,gaj86,GG96,GGH95,Glitzky,Glitzky_2}, the authors supplement the model with quite general reaction terms coming from mass-action kinetics chemistry subject to natural growth assumptions, and prove global well-posedness and exponential convergence to uniquely determined steady states.

Already the three-dimensional setting causes severe difficulties so that global existence (resp.\ well-posedness) has only been shown under additional assumptions so far. Typical examples for such further assumptions are:
\begin{enumerate}[label=$(\roman{*})$]
\item The initial data is sufficiently close to steady states, see \cite{BHN94};
\item The  a priori estimate $\sup_{t>0}\|c(t)\|_{L^2(\Omega)}<\infty$ holds, see e.g.~\cite{CL_multidim};
\item All charge numbers have the same sign, i.e.\ either $z_i\ge0$ or $z_i\le0$, see~\cite{GG96};
\item\label{ass:2:comp} There are only two components involved, i.e.~$P=2$, see~\cite{GG96};
\item The diffusional fluxes have a similar structure as above, with $\nabla c_i$ replaced by
the more general $c_i \, \nabla \mu_i$, where $\mu_i = \mu_i(c_i/c_i^\ast)$ is the chemical potential with a reference concentration $c_i^\ast$,
but the growth of the function $s\to \mu_i^{-1} (s)$ is at most like $s^{\gamma_i}$ with
$0\leq \gamma_i <\frac{4}{N-2}$ (see \cite{GS} and also \cite{GG96}). This does not include the less regular model studied here as soon as $N\geq 3$. 
In fact, this polynomial growth of
$\mu_i^{-1} (s)$ turns out to be satisfied for chemical potentials corresponding to so-called Fermi-Dirac
statistics. It is not satisfied in the case treated here which includes the important case of Boltzmann statistics with chemical potential of the type $\mu_i (s) = \mu_i^0 + RT \ln (c_i/c_i^\ast)$ and which potentially leads to quite less  regular solutions.

From the modelization point of view, let us emphasize that the Fermi-Dirac statistics is relevant for the transport of electrons in
metals or semiconductors, where quantum effects need to be accounted for, while the classical
Boltzmann statistics yields a very good description for the transport on ionic species in solutions (electrolytes).
\end{enumerate}

None of the above just stated assumptions will be imposed in our setup. Note also that most of the mathematical references cited above [except in particular \cite{GS}, \cite{GG96}]
consider the case of constant diffusivities $d_i$, which is not needed in our approach either. The advantage of this generality lies in the potential to tackle related quasilinear problems, where $d_i$ may also depend on $c$, by means of fixed-point methods, thus approaching more complete physical situations.

During the last decade, concerning the mathematical analysis of related models,
NPP has been complemented by the Navier-Stokes equations (NS) modeling the fluid flow; see, e.g., \cite{BFS12,DZC11,FS13,JS09,Ryham,Schmuck}. Partly due to the fact that NS itself is unknown to be well-posed in three-dimensional domains without further assumptions on the initial data, the results on NSNPP are similar to the ones in case of pure NPP. Without going into more detail on NSNPP here, let us point out that, apart from~\cite{BFS12,FS13}, assumption~\ref{ass:2:comp} is imposed throughout, which simplifies the situation considerably. This is why our approach to global well-posedness is also of interest to the more complicated situation of NSNPP with $P$ species.

Before stating the main results of this article, let us fix some conventions and notations. For convenience we set the parameters $F,R,T,\epsilon=1$; it is easy to check that our results remain true in the general case, where those parameters are positive constants. 

Throughout the paper, $\Omega$ denotes an open bounded and connected subset of  $\R^N$ with $C^2$-boundary $\pa\Omega$. 

Time-space cylinders are written as $Q_T=(0,T)\times\Omega$ and $\Sigma_T=(0,T)\times\pa\Omega$ for $T\in(0,\infty)$. 
We write $C^\infty_0(\R^N)$ for the space of smooth functions with compact support defined on $\R^N$. Note that $C^\infty(\overline\Omega)=\{v_{|\Omega}|v\in C^\infty_0(\R^N)\}$. Positive cones of nonnegative functions will be denoted by $L^p(\Omega)^+, C^\infty_0(\R^N)^+$, etc.,
and $\R^P_+$ is short for $(\R^+)^P$.

For the data to~(\ref{NPP1})-(\ref{NPP4}), we assume the following

\begin{equation}
\left.
\begin{array}{l}
\label{hyp:first} 
d_i\in L^\infty_{loc}([0,\infty);L^\infty(\Omega))\;{\rm  and \;for\; all\;} T>0,\\ \;{\rm there\; are\; numbers \;}\underline{d}(T), \overline{d}(T) \;{\rm such\; that\;}\\

0<\underline{d}(T)\leq d_i(t,x)\leq\overline{d}(T)<+\infty\mbox{ for }(t,x)\in Q_T.

\end{array}
\right\}
\end{equation}

\begin{equation}\label{hyp:f} 
\left.
\begin{array}{l}
f_i\in C([0,+\infty)\times\overline\Omega\times\R^P_+) \;{\rm with}\\
\hspace{0.3cm} (i)\; |f_i(t,x,y)|\leq C \;{\rm  for \;all\;} (t,x,y)\in[0,\infty)\times\overline\Omega\times\R^P,\\
\hspace{0.3cm} (ii)\; \pa_{y_j}f_i\in C([0,+\infty)\times\overline\Omega\times\R^P_+)
, j=1,\ldots,P,\\
\hspace{0.3cm} (iii)\;f_i(t,x,y)\ge0 \;{\rm if\;} y_i = 0, \;{\rm i.e. \;}f \;{\mbox{\rm is quasi-positive}}.
\end{array}
\right\}
\end{equation}

\begin{equation}\label{hyp:ic}
\left.
\begin{array}{l}
\hspace{0.2cm} (i)\;\tau\in L^\infty(\pa\Omega)^+\; {\rm with\;} \tau\not\equiv0;\\
\hspace{0.2cm} (ii)\; \xi \in L^2(\pa\Omega)\;{\mbox{\rm is a time-independent function;}}\\
\hspace{0.2cm} (iii)\; c^0\in L^p(\Omega)^+\;{\rm for\; some\;} p\in[2,\infty) \;{\rm with \;} p> N/2.
\end{array}
\right\}
\end{equation}

Instead of $f_i(t,x,y)$, we will often merely write $f_i(y)$.

Our first existence result then reads as follows.

\begin{theorem}\label{th1}
Under assumptions (\ref{hyp:first})-(\ref{hyp:ic}), there exists $(c,\Phi)\colon[0,\infty)\to [L^1(\Omega)^+]^N\times W^{1,2}(\Omega))$ such that (\ref{NPP1})-(\ref{NPP4}) is satisfied in the following sense:
\begin{enumerate}[label=$(\roman{*})$]
 \item For all $T>0$
\begin{align*}
&c_i\in C([0,T];L^1(\Omega))\cap L^1(0,T;W^{1,1}_{loc}(\Omega)),\\
&\Phi\in L^\infty(0,T;W^{1,2}(\Omega))\cap L^2(0,T;W_{loc}^{2,2}(\Omega)),\\
&d_i\nabla c_i +d_iz_ic_i\nabla \Phi\in L^1(Q_T).
\end{align*}
\item For all $\psi\in C^\infty(\overline{Q_T})$ such that $\psi(T)=0$
\begin{equation}\label{eq:th1:1}
 \int_{Q_T}-c_i\partial_t\psi +(d_i\nabla c_i +d_iz_ic_i\nabla \Phi)\nabla\psi =\int_{\Omega}c_i^0 \psi(0)+\int_{Q_T}f_i(c)\psi.
\end{equation}
\item For all $\varphi\in C^\infty(\overline{\Omega})$, for a.e. $t\in \R_+$,
\begin{equation}\label{eq:th1:2}
 \int_\Omega \nabla \Phi(t)\nabla \varphi +\int_{\partial\Omega}(\tau \Phi(t)-\xi)\varphi =\int_\Omega \left(\sum_{i=1}^P z_ic_i(t)\right)\varphi.
\end{equation}
\end{enumerate}
\end{theorem}

Let us sketch the main ideas of the proof.
It is well-known that there exists - at least formally - a nonnegative free energy functional which grows at most exponentially along solutions to the Nernst-Planck-Poisson system; it is given by
\[V_0(t)=\sum_{i=1}^N\int_\Omega \big( c_i\log c_i-c_i+1\big)
+\frac12\bigg(\int_\Omega|\nabla\Phi|^2+\int_{\pa\Omega}\tau|\Phi|^2\bigg).\]
From a physical viewpoint, $V$ can be interpreted as a modified Gibbs free energy with chemical and electrical contributions. Its formal dissipation rate can be estimated by
\begin{equation}\label{eq:diss:rate}
\Dt V_0(t)\le-\sum_{i=1}^N\int_\Omega\frac1{d_ic_i}|d_i\nabla c_i+d_iz_ic_i\nabla\Phi|^2+C\big(1+V_0(t)\big),
\end{equation}
where $C$ is some constant depending on the bounds of $f_i$. Remark that dissipation is present in case of non-vanishing mass fluxes $J_i=-d_i\nabla c_i-d_iz_ic_i\nabla\Phi$. Integrating the dissipation rate over $0$ to $T$ in time, we directly obtain a natural a priori bound for the solution~$(c,\Phi)$. In general, no further energy estimates are at hand; see e.g. \cite{CL_multidim}. Expanding the square under the integral then gives
\begin{equation}\label{eq:expl:diss}
\sum_{i=1}^N\int_{Q_T}d_i\left(\frac{|\nabla c_i|^2}{c_i}+z_i^2c_i|\nabla\Phi|^2+2z_i\nabla c_i\nabla\Phi\right)\leq C,
\end{equation}
where $C>0$ depends on $T$, the bounds on $f_i,\tau,\xi$ as well as the initial data $c^0$, since $V_0$ is bounded from below. Having a lower bound on $d_i$ this boundedness carries over to each summand without $d_i$. The first two terms are nonnegative and the third one can be treated by integration by parts:
\begin{equation}\label{eq:NPP:mot:3d}
\int_{Q_T}\sum_{i=1}^Pz_i\nabla c_i\nabla\Phi=-\int_{Q_T}\sum_{i=1}^Pz_ic_i\Delta\Phi+\int_{\Gamma_T}z_ic_i\pa_\nu\Phi.
\end{equation}
Since $-\Delta\Phi=\sum_iz_ic_i$, the first integral on the right-hand side has a positive sign. However, the boundary integral needs to be controlled, which is not possible with common absorption techniques in arbitrary space dimension.

In order to overcome this difficulty, we apply a cut-off technique, i.e.\ we multiply the nonnegative integrand of (\ref{eq:expl:diss}) by a nonnegative test function $\ze^2$ with compact support in order to get rid of the unpleasant boundary terms when integrating by parts. This procedure leads to local compactness inside $\Omega$ which we strongly use to prove existence. But we merely obtain the stated \emph{local} $W^{1,1}$-regularity in space for solutions $c_i$ to~(\ref{NPP1})-(\ref{NPP4}) in the general case.

Returning to the physical situation $N=3$, it turns out that the boundary integral in (\ref{eq:NPP:mot:3d}) can be handled appropriately if one assumes slightly more regularity on $\xi$ than stated in~(\ref{hyp:ic}). So in this case, there is global regularity in space for concentrations $c_i$, as stated in our second main result.
\begin{theorem}\label{th2}
Under assumptions (\ref{hyp:first})-(\ref{hyp:ic}) and $\xi\in L^q(\pa\Omega)$ with $q\in (2,\infty)$, there exists $(c,\Phi)\colon [0,\infty)\to [L^1(\Omega)^+]^N\times W^{1,2}(\Omega)$ such that, in addition to the assertions of Theorem~\ref{th1}, we have 
\begin{align*}
c_i&\in C([0,\infty);L^1 (\Omega))\cap L^1_{\rm loc}([0,\infty);W^{1,\frac32}(\Omega)),\\
\Phi&\in L^\infty_{\rm loc}([0,\infty);W^{1,2}(\Omega))\cap C([0,\infty);L^r(\Omega)),\  r\in[1,6).
\end{align*}
\end{theorem}

Of course, all those formal computations sketched above need to be carried out on approximate solutions having sufficient regularity. But detecting an appropriate approximate version of~(\ref{NPP1})-(\ref{NPP4}) is not straightforward. On the one hand, the resulting problem should possess a global solution and on the other hand, we shall not disturb the energy structure too much, because we want to use the just motivated computations. One possibility is to regularize the total charge density $\sum_iz_ic_i$ in the Poisson equation~(\ref{NPP3}) by resolvents of the Robin-Laplacian subject to a small parameter $\ep>0$. This is done in~\cite{fischer13,rolland12} (for the case $f\equiv0$); see also~\cite{FS13} for a corresponding approach to NSNPP. A Leray-Schauder fixed-point argument then provides global weak solutions for the approximate system and the formal energy relations are inherited by the resulting approximate solutions. More precisely, there is a modified functional $\tilde V$ consisting of two summands with the same growth properties as $V_0$. The first one has the same form as $V_0$ and the second one constitutes a ``small'' perturbation term. However, despite of the similarity of the energy structure, the resulting a priori estimates turn out to be quite technical.\saut
In this article, we pursue a different strategy which is inspired by~\cite{GG96},\cite{GS}. The idea is to consider a {\em quasi-linear} approximate version of the semi-linear problem~(\ref{NPP1})-(\ref{NPP4}). Let us explain this approach in more detail. For $\eta>0$ and $p\in(1,\infty)$, set $h(r)=r+\eta r^p$, $r\ge0$. Following \cite{GG96}, \cite{GS}, it can be shown that the system
\begin{align}
&\partial_tc_i +\textrm{div}(-d_i\nabla h(c_i) - d_iz_i c_i\nabla \Phi)=f_i(c),\qquad\ t>0,\ x\in\Omega,\label{h}\\
&\pa_\nu h(c_i)+z_ic_i\pa_\nu\Phi=0,\quad t>0,\ x\in\pa\Omega,\qquad c_i(0)=c_i^0,\quad x\in\Omega,\quad i=1,\ldots,P,\label{hb}
\end{align}
supplemented with (\ref{NPP3})-(\ref{NPP4}), provides a very similar energy structure and moreover,leads to a very regular $\nabla\Phi$. To be more specific, $V_0$ has to be modified in the following way:
\[V(t)=\sum_{i=1}^N\int_\Omega \psi(c_i)+\frac12\bigg(\int_\Omega|\nabla\Phi|^2+\int_{\pa\Omega}\tau|\Phi|^2\bigg),\]
where $\psi(r)=r\log r-r +1+\frac\eta{p-1} r^p\ge0$, $r\ge0$. It can still be shown that $V$ grows at most exponentially. The striking point with this setup is that we obtain an $L^\infty(0,T;L^p(\Omega))$-estimate on $c$, where we can choose $p$ as large as we please, so that the solution $\Phi$ of (\ref{NPP3})-(\ref{NPP4}) is very regular. This observation essentially allows us to solve in a strong sense problem~(\ref{h})-(\ref{hb}), (\ref{NPP3})-(\ref{NPP4}) via a Leray-Schauder fixed-point argument. The technical reason why this strategy works is the relation $h'(r)=r\psi''(r)$. Compactness of approximate solutions is much easier to establish as compared to~\cite{fischer13,rolland12}. 

\begin{remark}\label{rem:data}
Let us comment on the hypotheses we require for Theorems~\ref{th1} and~\ref{th2}.
\begin{enumerate}[label=$(\roman{*})$]
\item The authors are aware of the limited physical relevance of having bounded production terms $f_i$. In a forthcoming paper, however, we will use the results of this article in order to construct global solutions for~(\ref{NPP1})-(\ref{NPP4}), where, apart from certain natural structural assumptions, the right-hand sides $f_i$ are merely bounded in $L^1(Q_T)$ which includes lots of significant applications.
\item Time-independence of $\xi$ is present for technical reasons and simplicity. The $L^2$-regularity is used in order to estimate $\Phi(0)$ in $W^{1,2}(\Omega)$, since $V(0)$ plays an important role when estimating~(\ref{eq:expl:diss}). The fact that we need slightly more regularity on $\xi$ is related to the compactness of certain Sobolev embeddings and is commented in Remark~\ref{rem:3d}.
\item\label{rem:gen} It is easy to see that, in the situation of Theorem~\ref{th1}, it is actually possible to derive a corresponding result for the case when the boundary condition~(\ref{NPP4}) is replaced by
\[\Phi=\xi_D\quad\on\Gamma_D,\qquad\pa_\nu\Phi+\tau\Phi=\xi_R\quad\on\Gamma_R,\]
where $\Gamma_D\cup\Gamma_R=\pa\Omega$ and $\Gamma_D,\Gamma_R$ are disjoint, open and closed in $\pa\Omega$, and where $\xi_D,\xi_R\in L^2(\pa\Omega)$. However, it seems difficult to have such a generalization for Theorem~\ref{th2}, see Remark~\ref{rem:gen:3d}.
\end{enumerate}
\end{remark}

\begin{remark}
It is easy to see that in case of $f\equiv0$, we obtain $c\in L^\infty(0,+\infty;L^1(\Omega))$ and $\Phi\in L^\infty(0,+\infty;W^{1,2}(\Omega))$ in the statements of Theorems~\ref{th1} and~\ref{th2}, cf.\ Lemma~\ref{lem:entropy}.
\end{remark}


\section{An approximate system.}\label{S3}
The aim of this section is the existence of weak solutions on arbitrary large time intervals for an approximate version of (\ref{NPP1})-(\ref{NPP4}) as well as energy estimates. As already mentioned, the idea is taken from the nonlinear results in \cite{GG96},\cite{GS}. Since we are in a different setting and, moreover, to provide a self-contained presentation, we give a proof below.
It is based on the results in~\cite{LSU}, energy estimates, and Leray-Schauder's fixed-point theorem. The crucial point with our choice of approximating~(\ref{NPP1})-(\ref{NPP4}) is in particular the a priori estimate for $c$ in $L^\infty(0,T;L^p(\Omega))$ and the fact that the approximate system preserves the energy structure discussed above. \\

Let throughout this section $T_0>0$ and $\eta\in (0,1)$ be fixed and set $h(r):=r+\eta r^p$ for some $p\in [1,\infty)$. Using this notation, we state the following approximate version of system~(\ref{NPP1})-(\ref{NPP4}):
\begin{align}
\left.
\begin{array}{rcll}
 \partial_tc_i +\textrm{div}(-d_i\nabla h(c_i) - d_i z_i c_i\nabla \Phi)	&=&f_i(c) &\on Q_{T_0}    \\
\partial_\nu h(c_i)+z_ic_i\partial_\nu \Phi				&=&0 &\on \Sigma_{T_0}\\
c_i(0)									&=&c_i^{0} &\on \Omega
\end{array}
\right\},&\ i=1,\ldots,P, \label{approx:ci}\\
\left.
\begin{array}{rcll}
-\Delta \Phi								&=&\sum_{i=1}^Pz_ic_i			&\on Q_{T_0} 	\\
\partial_\nu \Phi+\tau \Phi 						&=&\xi			&\on\Sigma_{T_0}
\end{array}
\right\},&\label{approx:phi}
\end{align}
where we assume the following stronger conditions on the data as compared to~(\ref{hyp:first})-(\ref{hyp:ic}):
\begin{equation}\label{ass:add:1}
\left.
\begin{array}{l}
d_i\in C^2([0,+\infty)\times\overline\Omega;(0,\infty)) \;{\rm and \;for\;some} \;\underline d(T), \, \overline d(T)>0,\\
0<\underline d(T)\leq d_i(t,x)\leq\overline d(T)<\infty\quad\mbox{for all}\ (t,x)\in(0,T)\times\Omega,\\
f_i\in C^2([0,+\infty)\times\overline\Omega\times\R^P) \;{\rm with}\\
\hspace{0.2cm}(i)\; |f_i(t,x,y)|\leq C,  \forall (t,x,y)\in [0,\infty)\times\overline\Omega\times\R^P,\\
\hspace{0.2cm}(ii)\; f_i(t,x,y)\ge 0 \;{\rm if}\; y_i=0, \;{\rm i.e.}\; f_i\;{\mbox{\rm  is quasi-positive}},\\
\tau\in C^{1}(\partial\Omega)^+\; {\rm with}\; \tau\not\equiv 0,\\
\xi\in C^2(\pa\Omega),\\
c^{0}\in C^2(\Omega)^+.
\end{array}
\right\}
\end{equation}

\begin{remark}
From the uniform bound on $f_i$, we directly deduce $c\in L^\infty(0,T;L^1(\Omega)^+)$ for any $T>0$ by~(\ref{approx:ci}) using the no-flux boundary conditions:
\[\Dt\int_\Omega c_i=\int_\Omega f_i(c)\leq C.\]
We will frequently refer to this boundedness as \textit{mass control}.
\end{remark}

The solution of~(\ref{approx:ci})-(\ref{approx:phi}) will be based on a Leray-Schauder fixed-point argument. We gather some a priori estimates first. Formally the functional $V$ defined by
$$V_0(t)=\sum_{i=1}^P \int_\Omega c_i\log c_i-c_i+1 +\frac{1}{2} \bigg(\int_{\Omega} |\nabla \Phi|^2 +\int_{\partial\Omega} \tau\Phi^2\bigg),$$
can be shown to grow at most exponentially in time if $(c,\Phi)$ is a solution to the original problem~(\ref{NPP1})-(\ref{NPP4}) (cf.\ \cite{BFS12,GG96}). In the subsequent lemma, we show that this perturbation procedure does preserve this energetic structure. It is useful to define
\begin{equation}\label{eq:def:psi}
\psi(r)=r\log r-r+1+\frac{\eta}{p-1}r^p, \qquad r\ge0.
\end{equation}
Note that $\psi\geq 0$. Note also the important relation:
$h'(r)=r\psi''(r)$.

\begin{lemma}\label{lem:entropy}
Let~(\ref{ass:add:1}) hold, let $\si\in(0,1]$ and suppose that $c$ is a classical solution to~(\ref{approx:ci}) where $\Phi$ is a classical solution to
\[\begin{array}{rcll}
-\Delta \Phi								&=&\sigma\sum_{i=1}^Pz_ic_i			&\on Q_{T_0}, 	\\
\partial_\nu \Phi+\tau \Phi 						&=&\sigma\xi			&\on \Sigma_{T_0}.
\end{array}\]
Set $\widehat\Phi:=\si^{-1/2}\Phi$ and
\begin{align}\label{lem:entropy:lyapunov}
& V(t)=\sum_{i=1}^P \int_\Omega \psi(c)+\frac{1}{2} \bigg(\int_{\Omega} |\nabla \widehat\Phi|^2 + \int_{\partial\Omega} \tau\widehat\Phi^2\bigg),
\end{align}
where $\psi$ is given by~(\ref{eq:def:psi}). Then there exists a constant $C>0$ which is independent of $\eta>0$ and $\si\in(0,1]$ such that
\begin{enumerate}[label=$(\roman{*})$]
\item\begin{equation}\label{entropy}
\Dt V(t)\leq -\int_\Omega \sum_{i=1}^P\frac{1}{d_i c_i} |d_i \nabla h(c_i)+ d_i z_i c_i\nabla \Phi|^2+C\big(1+V(t)\big),
\end{equation}
\item \begin{equation}\label{entropy:2}
V(t)\leq C \quad \;{\rm for\;all}\;t\in[0,T_0],
\end{equation}
\item\begin{equation}\label{lem:entr:2:1bis}
 \sum_{i=1}^P\int_{Q_{T_0}}   \frac{1}{d_i c_i}|d_i\nabla h(c_i)+d_i z_ic_i\nabla \Phi|^2\leq C.
\end{equation}
\end{enumerate}
\end{lemma}

\begin{remark}\label{w12}
\rm Note that, since $\tau\not \equiv 0$, and since $\Omega$ is connected
$$\widehat{\Phi}\to \frac{1}{2} \bigg(\int_{\Omega} |\nabla \widehat\Phi|^2 + \int_{\partial\Omega} \tau\widehat\Phi^2\bigg),$$
defines a norm which equivalent to the usual norm on $W^{1,2}(\Omega)$. Since $\psi\geq 0$, it follows that a uniform bound on $V(t), t\in [0,T_0]$ provides a bound on $\widehat{\Phi}$ in $L^\infty\left(0,T_0; W^{1,2}(\Omega)\right)$. 
\end{remark}

\begin{proof}[Proof of Lemma~\ref{lem:entropy}]
We only provide a formal proof (since this will be used for regular enough solutions). Setting $J_i=-d_i(\nabla h(c_i)+z_i c_i\nabla \Phi)$, we have
\begin{align}
 \Dt\sum_{i=1}^P \int_\Omega \psi(c_i)	&=\sum_{i=1}^P \int_\Omega \psi'(c_i)\pa_tc_i=-\sum_{i=1}^P \int_\Omega \psi'(c_i)\divv J_i+\sum_{i=1}^P\int_\Omega\psi'(c_i)f_i(c)\nonumber\\
	&=\sum_{i=1}^P\int_\Omega\psi''(c_i)\nabla c_i\cdot J_i+\sum_{i=1}^P\int_\Omega f_i(c)
\big( \log c_i+\eta\frac{p}{p-1}c_i^{p-1} \big)\nonumber\\
	&\le\sum_{i=1}^P\int_\Omega\frac{h'(c_i)}{c_i}\nabla c_i\cdot J_i+C\Big(1+\sum_{i=1}^P\int_\Omega c_i\log c_i+\eta \, c_i^p\Big)\label{eq:entr:qp}\\
	&=\sum_{i=1}^P\int_\Omega\frac{\nabla h(c_i)+z_ic_i\nabla\Phi}{c_i}\cdot J_i-\int_\Omega\sum_{i=1}^Pz_i\nabla\Phi\cdot J_i+C\Big(1+\sum_{i=1}^P\int_\Omega\psi(c_i)\Big)\nonumber\\
	&\le-\sum_{i=1}^P \int_\Omega \frac{1}{d_i c_i}|J_i|^2 -\sum_{i=1}^P \int_\Omega z_iJ_i \cdot \nabla \Phi+C\Big(1+V(t)\Big).\label{eq:entr:der}
\end{align}
Note that in~(\ref{eq:entr:qp}), we made use of smoothness and quasi-positivity of $f_i$ in order to estimate $f_i(c)\log c_i\leq C$ for $c_i\leq1$, say. For the second term on the right-hand side of~(\ref{eq:entr:der}), we obtain by integration by parts
\begin{align*}
 -\sum_{i=1}^P 	\int_\Omega  z_iJ_i\cdot \nabla \Phi &=\sqrt\si\sum_{i=1}^P \int_\Omega z_i(\divv J_i)\widehat\Phi=-\sqrt\si\int_\Omega\left[ \partial_t\Big(\sum_{i=1}^P z_i c_i\Big)-\sum_{i=1}^Pz_if_i(c)\right] \widehat\Phi\\
& \leq \int_\Omega (\pa_t\Delta\widehat\Phi)\widehat\Phi+C\left[1+\int_\Omega\widehat{\Phi}^2\right],
\end{align*}
where we used the uniform bound on the $f_i$ and Young's inequality. Moreover we have
\begin{equation*}
\int_\Omega (\pa_t\Delta\widehat\Phi)\widehat\Phi=-\int_\Omega\nabla(\pa_t\widehat\Phi)\cdot\nabla\widehat\Phi-\int_{\pa\Omega}\tau(\pa_t\widehat\Phi)\widehat\Phi=-\frac12\Dt\left[\int_\Omega|\nabla\widehat\Phi|^2+\int_{\pa\Omega}\tau\widehat\Phi^2\right].
\end{equation*}
Finally, going back to (\ref{eq:entr:der}) and using the computations we just made, we deduce~(\ref{entropy}) (note also that $\int_\Omega\widehat{\Phi}^2\leq CV(t)$ according to Remark \ref{w12}. Gronwall's inequality then implies that $V$ grows at most exponentially in time. So the facts that $\tau\in L^\infty(\pa\Omega)^+$ and $c^0\in L^p(\Omega)^+$ and the estimate
\[\int_{\Omega} |\nabla \widehat\Phi(0)|^2 +\int_{\partial\Omega} \tau\widehat\Phi(0)^2\leq C\|\widehat\Phi(0)\|_{W^{1,2}(\Omega)}^2\leq C\sqrt\si(\|\sum_{i=1}^Pz_ic_i^0\|_{L^2(\Omega)}^2+\|\xi\|_{L^2(\pa\Omega)}^2)\]
imply~(\ref{entropy:2}). Relation~(\ref{lem:entr:2:1bis}) is then a consequence of integrating~(\ref{entropy}) from $0$ to $T_0$.
\end{proof}

Let us now proceed with the construction of a global solution to~(\ref{approx:ci})-(\ref{approx:phi}) through a fixed-point method.\saut
\nd\textbf{Definition of the fixed-point map.} Let~(\ref{ass:add:1}) be satisfied. Set $X:=L^\infty(0,T_0;W^{1,\infty}(\Omega))$ and let $\Phi\in X$. The idea is to define $c$ as the solution to~(\ref{approx:ci}) and then $\TT\Phi:=\widehat\Phi\in X$ as the solution to~(\ref{approx:phi}) with data $c$. A fixed point of this map $\TT$ is then a solution to~(\ref{approx:ci})-(\ref{approx:phi}) on $Q_{T_0}$. For regularity reasons, we do not define $\TT$ directly this way:  we first rely on some approximation and truncation arguments. \saut
In this respect, we first consider a smooth approximation $(\Phi^k)_{k\in\N}$ of $\Phi$ such that $\nabla\Phi^k\to\nabla\Phi$ pointwise and $\|\Phi^k\|_{L^\infty(0,T_0;W^{1,\infty}(\Omega))}\leq\|\Phi\|_{L^\infty(0,T_0;W^{1,\infty}(\Omega))}$. Next, we replace $h(r)=r+\eta r^p$ by $h_M(r):=r+\eta T_M(r)$ where $T_M$ is a regular bounded nondecreasing approximation of $\inf\{r^p,M\}$ (thus $h'_M(r)\geq 1$). Then the problem
\begin{equation*}
\left.
\begin{array}{rcll}
 \partial_tc_i +\textrm{div}(-d_i\nabla h^M(c_i) - d_i z_i c_i\nabla \Phi^k)	&=&f_i(c) &\on Q_{T_0},    \\
\partial_\nu h^M(c_i)+z_ic_i\partial_\nu \Phi^k			&=&0 &\on \Sigma_{T_0},\\
c_i(0)									&=&c_i^{0} &\on \Omega,
\end{array}\right\},\ i=1,\ldots,P,
\end{equation*}
has a unique nonnegative solution $c^{M,k}=(c_1^{M,k},\ldots,c_P^{M,k})\in H^{2+\be,1+\be/2}(\overline{Q_{T_0}})$ for some $\be>0$ by~\cite[Theorem 7.4, page 491]{LSU}. Moreover, the norms of $c^{M,k}$ in $L^\infty(Q_{T_0})$ and in $L^2(0,T_0;W^{1,2}(\Omega))$ depend only on the initial data, the bounds $\underline d(T_0),\overline d(T_0)$, and $\|\nabla\Phi^k\|_{L^\infty(Q_{T_0})}$. This can be seen by the usual technique of multiplication of the system by $(c_i^{M,k})^{q-1}$, $q=2^m$, where $m=1,2,\ldots$; see, e.g., \cite{BFS12,CL_multidim}.  But for completeness, we give an explicit proof in the Appendix (see Lemma \ref{supbound}) based on a classical technical lemma from \cite{LSU}.

Hence, choosing $M>0$ large enough, we obtain a solution $c^k=(c_1^k,\ldots,c_P^k)$ for the problem
\begin{equation}\label{eq:LS:cc}
\left.
\begin{array}{rcll}
 \partial_tc_i +\textrm{div}(-d_i\nabla h(c_i) - d_i z_i c_i\nabla \Phi^k)	&=&f_i(c) &\on Q_{T_0},    \\
\partial_\nu h(c_i)+z_ic_i\partial_\nu \Phi^k				&=&0 &\on \Sigma_{T_0},\\
c_i(0)									&=&c_i^{0} &\on \Omega,
\end{array}\right\},\ i=1,\ldots,P.
\end{equation}
Since $c^k$ is a classical solution to~(\ref{eq:LS:cc}), it also satisfies the weak formulation
\begin{equation}\label{eq:fix:weak}
\int_{Q_{T_0}}-c_i^k\psi_t+(d_i\nabla h(c_i^k)+d_iz_ic_i^k\nabla\Phi^k)\nabla\psi=\int_\Omega c_i^0\psi(0)+\int_{Q_{T_0}}f_i(c^k)\psi,
\end{equation}
for all $\psi\in C^\infty(\overline{Q_{T_0}})\mbox{ with }\psi(T)=0$ and for all $i=1,\ldots,P$.

The bounds on $c^k$ in $L^\infty(Q_{T_0})\cap L^2(0,T_0;W^{1,2}(\Omega))$  and on $\pa_tc^k$ in $L^2(0,T_0;W^{-1,2}(\Omega))$ (see Lemma \ref{H-1} in the Appendix), imply that $c^k$ is relatively compact in $L^2(Q_{T_0})$ by virtue of Aubin-Simon compactness; cf.\ \cite[Corollary 4]{simon86}. Thus we may assume that $c^k\to c$ strongly in $L^2(Q_{T_0})$ and $\nabla c^k\to \nabla c$ weakly in $L^2(Q_{T_0})$. The limit $k\to\infty$ in~(\ref{eq:fix:weak}) shows that $c$ is a weak solution to
\begin{equation}\label{eq:LS:c}
\left.
\begin{array}{rcll}
 \partial_tc_i +\textrm{div}(-d_i\nabla h(c_i) - d_i z_i c_i\nabla \Phi)	&=&f_i(c) &\on Q_{T_0},    \\
\partial_\nu h(c_i)+z_ic_i\partial_\nu \Phi			&=&0 &\on \Sigma_{T_0},\\
c_i(0)									&=&c_i^{0} &\on \Omega,
\end{array}\right\},\ i=1,\ldots,P,
\end{equation}
in the sense of~(\ref{eq:fix:weak}). Let us show that it is unique in the class $L^\infty(Q_{T_0})^+\cap L^2(0,T_0;W^{1,2}(\Omega))$ through the following {\em dual approach} lemma.

\begin{lemma}\label{lem:dual}
Let $T\in(0,\infty)$, $\Theta\in C_0^\infty(Q_T)$ and suppose $c,\widehat c\in L^\infty(Q_{T})^+\cap L^2(0,T;W^{1,2}(\Omega))$ are two solutions to~(\ref{eq:LS:c}) for a given $\Phi$. Then there is a solution to the problem
\begin{equation}\label{eq:dual}
\left.
\begin{array}{l}
-(\pa_t\Psi_i+A_i\divv(d_i\nabla\Psi_i)-d_iz_i\nabla\Phi\nabla\Psi_i+\sum_{j=1}^PB_{ji}\Psi_j)=\Theta_i\on Q_{T},\\
\pa_\nu\Psi_i(T)=0\on\Sigma_{T},\qquad\Psi_i(T)=0\on\Omega,\\
\Psi_i\in L^2(0,T;W^{2,2}(\Omega))\cap W^{1,2}(0,T;L^2(\Omega)),
\end{array}\right\},\ i=1,\ldots,P,
\end{equation}
where $A_i=\frac{h(c_i)-h(\widehat c_i)}{c_i-\widehat c_i}$, $B_{ij}=\int_0^1\pa_{y_j}f_i(sc+(1-s)\widehat c)ds$.
\end{lemma}
\begin{remark} \rm The form of $B_{ij}$ in this lemma originates from the elementary relation
\[f_i(y)-f_i(\widehat y)=\int_0^1\Dt\big(f_i(sy+(1-s)\widehat y)ds=\sum_{j=1}^P\int_0^1\pa_{y_j}f_i(sy+(1-s)\widehat y)(y_j-\widehat y_j)ds.\]
\end{remark}

\begin{proof}[Proof of Lemma~\ref{lem:dual}]
Note that due to the monotonicity of $h$ and the boundedness of $c,\widehat c$, there are positive constants $a,b$ such that $a\le\frac{h(c_i)-h(\widehat c_i)}{c_i-\widehat c_i}\le b$; moreover $B_{ij}\in L^\infty(Q_T)$. The first step for the construction of a solution $\Psi$ consists in suitably regularizing $A_i$ and $B_{ij}$ so that classical results on linear parabolic systems yield classical solutions on $Q_T$ (see e.g.~\cite{denk,dhp}). Then a priori estimates in $W^{1,2}(0,T;L^2(\Omega))\cap L^2(0,T;W^{2,2}(\Omega))$ are obtained by multiplying~(\ref{eq:dual}) with $\Psi_i$ and $\Delta\Psi_i$, integrating over $Q_T$, and summing over $i$ respectively. This computation actually works very similarly to the proof of Lemma~5.2 in~\cite{LPR}, which is why we omit the details here.
\end{proof}

So let $c,\widehat c\in L^\infty(Q_{T})^+\cap L^2(0,T;W^{1,2}(\Omega))$ be two solutions to~(\ref{eq:LS:c}), let $\Theta=(\Theta_1,\ldots,\Theta_P)\in C_0^\infty(Q_{T_0})$ be arbitrary, and suppose that $\Psi=(\Psi_1,\ldots,\Psi_P)$ is a solution to~(\ref{eq:dual}). Then using integration by parts, we compute
\begin{align*}\sum_{i=1}^P&\int_{Q_{T_0}}(c_i-\widehat c_i)\Theta_i\\
	&=\sum_{i=1}^P\int_{Q_{T_0}}-(c_i-\widehat c_i)\pa_t\Psi_i+\big(d_i\nabla(h(c_i)-h(\widehat c_i))+d_iz_i(c_i-\widehat c_i)\nabla\Phi\big)\nabla\Psi_i-\sum_{j=1}^PB_{ij}\Psi_i\\
	&=\sum_{i=1}^P\int_{Q_{T_0}}-(c_i-\widehat c_i)\pa_t\Psi_i+\big(d_i\nabla(h(c_i)-h(\widehat c_i))+d_iz_i(c_i-\widehat c_i)\nabla\Phi\big)\nabla\Psi_i-(f_i(c)-f_i(\widehat c))\Psi_i=0,
\end{align*}
whence $c=\widehat c$.

From the fact that $c\in L^\infty(Q_{T_0})$, we can then finally define
\begin{equation}\label{def:fix:map}
\TT\Phi:=\widehat\Phi\in L^\infty(0,T_0;W^{1,\infty}(\Omega))
\end{equation}
as the solution to
\begin{equation}\label{eq:LS:phi}
\left.
\begin{array}{rcll}
-\Delta \widehat\Phi								&=&\sum_{i=1}^Pz_ic_i			&\on Q_{T_0},	\\
\partial_\nu \widehat\Phi+\tau \widehat\Phi 						&=&\xi			&\on \Sigma_{T_0},
\end{array}\right\}
\end{equation}
see e.g.~\cite{grisvard}.

\begin{lemma}\label{lem:prop}
Let the data of problem~(\ref{approx:ci})-(\ref{approx:phi}) satisfy (\ref{ass:add:1}) and let $p\in[2,\infty)$ with $p> N/2$. Then there exists $(c,\Phi)$ such that $ c\in L^\infty(Q_{T_0})\cap L^2(0,T_0;W^{1,2}(\Omega))$, $\partial_t c \in L^2(0,T_0;W^{-1,2}(\Omega))$, $\Phi\in L^\infty(0,T_0;W^{2,p}(\Omega))$ and $(c,\Phi)$ satisfies (\ref{approx:ci})-(\ref{approx:phi}) where \eqref{approx:ci} is satisfied in the sense that, for all $\psi\in C^\infty(\overline{Q_{T_0}})$ with $\psi(T_0)=0,$
\begin{equation}\label{eq:prop1:main}
 \int_{Q_{T_0}}-c_i\partial_t\psi +(d_i\nabla h(c_i) +d_iz_ic_i\nabla \Phi)\nabla \psi =\int_{\Omega} c_i^0 \psi(0)+\int_{Q_{T_0}}f_i(c)\psi
\end{equation}
and \eqref{approx:phi} is satisfied in a pointwise sense.
\end{lemma}
\begin{proof}
It is sufficient to show that $\TT$, as defined in~(\ref{def:fix:map}), has a fixed point. In order to do so we will use the Leray-Schauder fixed-point theorem. \saut
Let us first show that $\TT$ maps bounded sets into relatively compact ones. To this end, suppose $(\Phi^n)_{n\in\N}$ is a bounded sequence in $X$ and $(c^n,\widehat \Phi^n)$ is the corresponding solution of (\ref{eq:LS:c}), (\ref{eq:LS:phi}). As indicated in the construction of the solution to (\ref{def:fix:map}), the $L^2(0,T_0;W^{-1,2}(\Omega))$-norm of $\pa_tc^n$ and the $L^\infty(Q_{T_0})$-norm of $c^n$ only depend on the $L^\infty(0,T_0;W^{1,\infty}(\Omega))$-norm of $\Phi^n$. Then, by differentiating (\ref{eq:LS:phi}) with respect to $t$, we see that $\pa_t\widehat \Phi^n$ is bounded in $L^2(0,T_0;W^{1,2}(\Omega))$ and $\widehat \Phi^n$ is bounded in $L^\infty(0,T_0;W^{2,q}(\Omega))$ for any $q<\infty$. Then from~\cite[Corollary~4]{simon86}, it follows that $(\widehat \Phi^n)$ is relatively compact in $X$, whence the compactness of $\TT$.
\saut
To prove continuity of $\TT$, let $\Phi^n\rightarrow \Phi$ in $X$. As a consequence $\{\widehat\Phi^n=\TT\Phi^n,\;n\in\N\}$ is relatively compact in $X$. Let $\widehat\Phi$ be a limit point. Similarly as before, the estimates  sketched above and the results in \cite{simon86} guarantee that $(c^n)_{n\in\N}$ is bounded in $L^\infty(Q_{T_0})\cap L^2(0,T_0;W^{1,2}(\Omega))$ and relatively compact in $L^2(Q_{T_0})$. Therefore, we may extract a subsequence that converges a.e.\ and in any $L^q(Q_{T_0})$ for $q<+\infty$ to a limit $c$, and such that $\nabla c^n\rightarrow \nabla c$ weakly in $L^2(Q_{T_0})$. Then we pass to the limit $n\rightarrow +\infty$ in (\ref{eq:LS:c}) and using uniqueness, $c$ is {\it the} solution in $L^\infty(Q_{T_0})\cap L^2(0,T_0;W^{1,2}(\Omega))$ of (\ref{eq:LS:c}) with data $\Phi$. Then we pass to the limit $n\rightarrow +\infty$ in equation (\ref{eq:LS:phi}), which yields $\widehat \Phi =\TT\Phi$. The only possible limit point for $(\TT\Phi^n)_{n\in\N}$ is $\TT\Phi$ and $(\TT\Phi^n)_{n\in\N}$ lies in a compact subset of $X$, so $\TT\Phi^n\rightarrow \TT\Phi$, whence the continuity of $\TT$.
\saut
For the a priori estimate, let $\la \in(0,1]$, $\Phi\in X$, and let $(c,\TT\Phi)$ be the corresponding solution of (\ref{approx:ci})-(\ref{approx:phi}). Assume $\Phi=\la \TT \Phi$. From Lemma~\ref{lem:entropy}, we can bound the $L^\infty(0,T_0;L^p(\Omega))$-norm of $c$ independently of $\la$. As a consequence, the $L^\infty(0,T_0;W^{2,p}(\Omega))$-norm of $\TT\Phi$ is bounded independently of $\la$ via~(\ref{eq:LS:phi}). Because of the embedding $W^{2,p}(\Omega)\inj W^{1,\infty}(\Omega)$ (recall that $p>N/2$), $\Phi$ is bounded in $X$. Therefore, any solution of $\Phi=\la \TT \Phi$ is a priori bounded in $X$. According to Leray-Schauder's theorem, $\TT$ has a fixed point $\Phi$ and the corresponding $(c,\Phi)$ satisfies (\ref{approx:ci})-(\ref{approx:phi}) in the sense of Lemma~\ref{lem:prop}.
\end{proof}


The energy estimates contained in the following lemma will allow us to pass to the limit as $\eta\rightarrow 0$ in (\ref{approx:ci})-(\ref{approx:phi}) in the proofs of Theorems~\ref{th1} and~\ref{th2}.

\begin{lemma}\label{lem:entr:2}
Let the premises of Lemma~\ref{lem:entropy} be satisfied with $\si=1$.
\begin{enumerate}[label=$(\roman{*})$]
\item For $\zeta\in C^\infty_c(\Omega,\R_+)$, there exists $C=C(\zeta,T_0)>0$ such that
\begin{equation}\label{lem:entr:2:2}
 \int_{Q_{T_0}} \left(\frac{|\nabla c_i|^2}{c_i}\zeta^2 + \frac{|\eta\nabla c_i^p+z_ic_i\nabla\Phi|^2}{c_i}+\eta|\nabla c_i^{p/2}|^2+|\Delta\Phi|^2\right)\zeta^2 \leq C.
\end{equation}
\item If $N\leq3$ and if, in addition, $\xi\in L^q(\pa\Omega)$ for some $q>2$, then there is a constant $C=C(T_0)>0$ such that
\begin{equation}\label{eq:lem:diss}
 \int_{Q_{T_0}} \frac{|\nabla c_i|^2}{c_i} + \frac{|\eta\nabla c_i^p+z_ic_i\nabla\Phi|^2}{c_i}+\eta|\nabla c_i^{p/2}|^2+|\Delta\Phi|^2 \leq C.
\end{equation}
\end{enumerate}
\end{lemma}
\begin{proof}
Throughout the proof, $C$ always denotes a positive constant that may depend on $T_0,\zeta$ and on the data of $(\ref{approx:ci})$-$(\ref{approx:phi})$, but not on $\eta$. \saut
To prove $(i)$, recall that $d_i$ is bounded from below and from above on $Q_{T_0}$ by positive constants $\underline d(T_0),\overline d(T_0)$ and $\zeta \in L^\infty(\Omega)^+$, so there exists $C=C(T_0)>0$ such that
\begin{align}
C&\ge \sum_{i=1}^P\int_{Q_{T_0}} \frac{1}{c_i}|\nabla c_i+\eta\nabla c_i^p+z_ic_i\nabla \Phi|^2 \zeta^2\nonumber\\
	&=\sum_{i=1}^P\int_{Q_{T_0}} \frac{|\nabla c_i|^2}{c_i}\zeta^2+\frac{|\eta\nabla c_i^p+z_ic_i\nabla \Phi|^2}{c_i}\zeta^2+2\frac{\nabla c_i}{c_i}(\eta\nabla c_i^p+z_ic_i\nabla \Phi)\ze^2\nonumber\\
	&=\sum_{i=1}^P\int_{Q_{T_0}} \frac{|\nabla c_i|^2}{c_i}\zeta^2+\frac{|\eta\nabla c_i^p+z_ic_i\nabla \Phi|^2}{c_i}\zeta^2+\frac8p\eta|\nabla c_i^{p/2}|^2\ze^2+2z_i\nabla c_i\cdot \nabla\Phi\ze^2.\label{eq:entr:I}
\end{align}
It is sufficient to show that $\sum_i\int_{Q_{T_0}}z_i\nabla c_i\cdot \nabla \Phi\zeta^2$ is bounded from below. To this end, we employ integration by parts and equation~(\ref{approx:phi}) to obtain
\begin{align}
 \sum_{i=1}^P\int_{Q_{T_0}}z_i\nabla c_i\cdot\nabla \Phi\zeta^2&	= -\int_{Q_{T_0}}\sum_{i=1}^Pz_ic_i\Delta \Phi\zeta^2-\int_{Q_{T_0}}\sum_{i=1}^Pz_ic_i\nabla \Phi\cdot \nabla\zeta^2,\nonumber\\
								&\ge\int_{Q_{T_0}}|\Delta\Phi|^2 \zeta^2-\Big|\int_{Q_{T_0}}\Delta\Phi(\nabla \Phi\cdot\nabla\zeta^2)\Big|.\label{eq:entr:II}
\end{align}
For the last term, we then compute with Young's inequality
\begin{equation}\label{eq:entr:III}
\left|\int_{Q_{T_0}}\Delta\Phi(\nabla \Phi\cdot\nabla\zeta^2)\right|\leq\frac12\int_{Q_{T_0}}|\Delta\Phi|^2\zeta^2+C\int_{Q_{T_0}}|\nabla\Phi|^2\frac{|\nabla\zeta^2|^2}{\ze^2}\leq \frac12\int_{Q_{T_0}}|\Delta\Phi|^2\zeta^2+C
\end{equation}
for some $C>0$, where we used Lemma \ref{lem:entropy} for the uniform bound of $\nabla\Phi$ in $L^2(\Omega)$. Combining (\ref{eq:entr:I})-(\ref{eq:entr:III}) yields~(\ref{lem:entr:2:2}).\saut

For the proof of~$(ii)$, recall that $4|\nabla\sqrt{c_i}|^2=\frac{|\nabla c_i|^2}{c_i}$ and $0<\underline d(T_0)\leq d_i\leq\overline d(T_0)<\infty$ in $Q_{T_0}$. In contrast to the proof of $(i)$, we omit multiplication with a test function $\zeta^2$ and expand (\ref{lem:entr:2:1bis}) directly to obtain
\begin{equation}\label{eq:diss:est:1}
\sum_{i=1}^P\int_{Q_{T_0}}4|\nabla \sqrt{c_i}|^2+\frac{|\eta\nabla c_i^p+z_ic_i\nabla \Phi|^2}{c_i}+\frac8p\eta|\nabla c_i^{p/2}|^2+2z_i\nabla c_i \cdot \nabla\Phi\leq C.
\end{equation}
Here, we integrate by parts the last term on the left-hand side
\begin{equation}\label{eq:diss:est:2}
\int_{Q_{T_0}}\sum_{i=1}^Pz_i\nabla c_i\cdot \nabla\Phi= -\int_{Q_{T_0}}\sum_{i=1}^Pz_ic_i\Delta\Phi+\sum_{i=1}^P\int_{\Sigma_{T_0}}z_ic_i\pa_\nu\Phi=\int_{Q_{T_0}}|\Delta\Phi|^2+\int_{\Sigma_{T_0}}\sum_{i=1}^P z_ic_i(\xi-\tau\Phi).
\end{equation}
We will now prove that the last boundary integral can be estimated in dimension $N=3$.

Recall for the following that the map $u\in W^{1,2}(\Omega)\to u_{|_{\partial\Omega}}\in L^r(\partial\Omega)$ is continuous for $r=4$ and compact for $r<4$ so that, if $r<4$, for all $\epsilon>0$, there exists $C=C(\epsilon)$ such that
\begin{equation}\label{W12embedd}
\forall w\in W^{1,2}(\Omega),\; \|w\|_{L^r(\partial\Omega)}\leq \epsilon\|\nabla w\|_{L^2(\Omega)}+C\|w\|_{L^2(\Omega)}.
\end{equation}
We can estimate from below the $c_i\xi$-terms of (\ref{eq:diss:est:2}) by H\"{o}lder's inequality as follows
$$\int_{\Sigma_{T_0}}c_i\xi\geq -\int_{\Sigma_{T_0}}|c_i\xi|\geq -\|\sqrt{c_i}\|^2_{L^2(0,T_0;L^r(\pa\Omega))}\|\xi\|_{L^q(\pa\Omega)},$$
where $r=2q/(q-1)<4$ (recall that $\xi\in L^q(\partial\Omega), q>2$). Applying (\ref{W12embedd}) to $w=\sqrt{c_i}$ and using the mass control, we have for all $\alpha>0$
\begin{equation}\label{eq:est:int:xi}
\int_{\Sigma_{T_0}}c_i\xi\geq -\alpha\|\nabla \sqrt{c_i}\|_{L^2(Q_{T_0})}^2-C\|\sqrt{c_i}\|_{L^2(Q_{T_0})}^2\geq-\alpha \|\nabla \sqrt{c_i}\|_{L^2(Q_{T_0})}^2-C.
\end{equation}
Similarly for the $c_i\Phi$-term of (\ref{eq:diss:est:2}), we use H\"{o}lder's inequality with $r=8/3, q=4$ and we remember that, by  Lemma~\ref{lem:entropy}, $\|\Phi\|_{L^\infty(0,T_0;W^{1,2}(\Omega))}$ is bounded independently of $\eta$. We then obtain
\begin{align}
\int_{\Sigma_{T_0}} c_i\Phi&\geq -\int_{\Sigma_{T_0}}|c_i\Phi|\geq-\|\sqrt{c_i}\|_{L^2(L^{8/3}(\pa\Omega))}^2\|\Phi\|_{L^\infty(L^4(\pa\Omega))}\nonumber\\
	&\geq-C\|\sqrt{c_i}\|_{L^2(L^{8/3}(\pa\Omega))}^2\|\Phi\|_{L^\infty(W^{1,2}(\Omega))} \geq-\alpha \|\nabla \sqrt{c_i}\|_{L^2(Q_{T_0})}^2-C\|\sqrt{c_i}\|_{L^2(Q_{T_0})}^2\nonumber\\
	&\geq -\alpha \|\nabla \sqrt{c_i}\|_{L^2(Q_{T_0})}^2-C.\label{eq:diss:est:4}
\end{align}
Finally, by choosing $\alpha$ small enough, (\ref{eq:diss:est:1}), (\ref{eq:diss:est:2}), (\ref{eq:est:int:xi}),  (\ref{eq:diss:est:4}) give~(\ref{eq:lem:diss}).

\end{proof}

\begin{remark}\label{rem:3d}
The crucial ingredient of the proof of Lemma~\ref{lem:entr:2} $(ii)$ is the compactness embedding (\ref{W12embedd}). In order to use an appropriate compactness argument for (\ref{eq:est:int:xi}), the adjustment of the integrability of $\xi$ is needed in~$(ii)$. If one pursues the same strategy for controlling the boundary integral in~(\ref{eq:diss:est:4}) in dimension 4, the corresponding embedding becomes sharp, thus no obvious absorption as in (\ref{eq:diss:est:4}) seems possible for $N>3$.
\end{remark}

\begin{remark}\label{rem:gen:3d}
In view of the generalization mentioned in Remark~\ref{rem:data}~\ref{rem:gen}, note that the incorporation of a Dirichlet boundary part does not change anything in the proof of Lemma~\ref{lem:entr:2}~$(i)$, whereas for~$(ii)$ it is necessary to have information on $\pa_\nu\Phi$ on the whole boundary $\pa\Omega$ (cf.\ equation~(\ref{eq:diss:est:2})).
\end{remark}

\section{Proof of Theorem \ref{th1}.  }\label{S4}
Let us introduce two sequences $T^n\to\infty$ and $\eta^n\searrow 0$ as $n\to\infty$. For $n\in\N$, let $d_i^n$, $f_i^n$, $\tau^n$, $\xi^n$ and $c^{0n}$ satisfy conditions~(\ref{ass:add:1}) such that $d_i^n(t,x)\to d_i(t,x)$ almost everywhere with $\inf d_i\leq d_i^n\leq \sup d_i$, $f_i^n(t,x,y)\to f_i(t,x,y)$ uniformly on compact sets, $c_i^{0n}\to c_i^0$ in $L^2(\Omega)$, $\tau^n\to \tau$ a.e.\  with $\|\tau^n\|_{L^{\infty}(\partial\Omega)}\leq \|\tau\|_{L^\infty(\partial\Omega)}$ and $\xi^n\to\xi$ in $L^2(\pa\Omega)$. We denote by $(c^n,\Phi^n)$ a solution of $(\ref{approx:ci})$-$(\ref{approx:phi})$ on $Q_{T^n}$ with parameters $p$ and $\eta^n$ and data $d_i^n$, $f_i^n$, $\tau^n$, $\xi^n$, $c^{0n}$ and we write $J_i^n =-d_i^n\nabla h(c_i^n)-d_i^nz_ic_i^n \nabla \Phi^n$.

We will show that, up to a subsequence, $(c_i^n,\Phi^n)$ converges in an appropriate sense to a solution of the limit problem and in such a way that $J_i^n$ converges weakly in $L^1(Q_T)$ to $-d_i[\nabla c_i+z_ic_i\nabla\Phi]$.\\

\noindent{\bf Step 1: convergence of $c^n,\Phi^n$.} From mass control, $ \sqrt{c_i^n}$ is bounded in $L^\infty(0,T;L^2(\Omega))$ for all $T\in (0,\infty)$. From Lemma~\ref{lem:entr:2}~$(i)$,

\begin{equation}\label{sqrt(c)}
2\nabla \sqrt{c_i^n}=\frac{\nabla c_i^n}{\sqrt{c_i^n}} \mbox{ is bounded in }L^2(0,T;L^2_{loc}(\Omega))\;{\rm for \;all}\; T\in (0,\infty) .\end{equation}

Let us consider the sequence $\sqrt{c_i^n+1}$. Recall that $J_i^n/\sqrt{c_i^n}$ is bounded in $L^2(Q_T)$ by Lemma \ref{lem:entropy}~$(iii)$, and so is $J_i^n/\sqrt{c_i^n+1}$. It follows that
\begin{align}
2\partial_t \sqrt{c_i^n+1}&=\frac{\partial_t c_i^n}{\sqrt{c_i^n+1}}=\frac{-\divv J_i^n}{\sqrt{c_i^n+1}}+\frac{f_i^n(c^n)}{\sqrt{c_i^n+1}}\nonumber\\
	&=-\divv\Big(\frac{J_i^n}{\sqrt{c_i^n+1}}\Big)-\frac{J_i^n\cdot\nabla c_i^n}{2(c_i^n+1)^{3/2}}+\frac{f_i^n(c^n)}{\sqrt{c_i^n+1}}\nonumber\\
&=-\divv\Big(\frac{J_i^n}{\sqrt{c_i^n+1}}\Big)-\frac1{2\sqrt{c_i^n+1}}\frac{J_i^n}{\sqrt{c_i^n+1}}\cdot\frac{\nabla c_i^n}{\sqrt{c_i^n+1}}+\frac{f_i^n(c^n)}{\sqrt{c_i^n+1}},\label{eq:comp:ci:t}
\end{align}
hence $\partial_t \sqrt{c_i^n+1}$
is bounded in $L^1\big(0,T;W^{-1,2}(\Omega)+L^1_{loc}(\Omega))\big)$ for all $T\in (0,\infty)$ (using Schwarz inequality for the middle term). Since $\sqrt{c_i^n+1}$ is bounded in $L^2(0,T;W^{1,2}_{loc}(\Omega))$ for all $T\in (0,\infty)$, we deduce with~\cite{simon86} that $\sqrt{c_i^n+1}$ is relatively compact in $L^2(K)$ for all compact subset $K\subset [0,\infty)\times\Omega$. By a standard diagonal process, this provides compactness in $L^2_{loc}([0,\infty);L^2_{loc}(\Omega))$ and therefore $c_i^n$ is relatively compact in $L^1_{loc}([0,\infty);L^1_{loc}(\Omega))$. We may assume that, up to a subsequence, $c^{n}$ converges  a.e.\ in $(0,\infty)\times \Omega$ as well. According to Lemma~\ref{lem:entropy}, $c_i^{n}\log c_i^{n}$ is bounded in $L^\infty(0,T;L^1(\Omega))$; so Vitali's theorem guarantees that $c_i^n$ actually converges in $L^1(Q_T)$ for all $T\in (0,\infty)$. In particular, up to a subsequence again, we may assume that $c^n(t)$ converges in $L^1(\Omega)$ for a.e.\ $t$. Moreover, for some $M<\infty$,
its limit $c$ satisfies
\begin{equation}\label{clogc}
\forall i=1,...,P,\; a.e.\,t\in (0,T),\; \;\int_\Omega c_i(t)|\log c_i(t)|\leq M.
\end{equation}
And the solution $\Phi^n$ of
$$-\Delta \Phi^n(t)=\sum_{i=1}^Pz_ic_i^n(t),\;\;\partial_\nu\Phi^n+\tau^n\Phi^n=\xi^n,$$
which is bounded in $L^\infty(0,T;W^{1,2}(\Omega))$, converges for a.e.\ $t$, weakly in $W^{1,2}(\Omega)$, to the solution $\Phi$ of the expected limit problem in the variational sense (\ref{eq:th1:2}). Moreover, this convergence holds also weakly in $L^q(0,T;W^{1,2}(\Omega))$ for any $q<+\infty$ and $T<+\infty$. Note also that $\Phi\in L^2(0,T;W^{2,2}_{loc}(\Omega))$ by (\ref{lem:entr:2:2}) in Lemma \ref{lem:entr:2}.\\

\noindent{\bf Step 2: Convergence of $J_i^n$ to $-d_i(\nabla c_i+z_ic_i\nabla\Phi)$.} Since, on one hand, $J_i^n/\sqrt{c_i^n}$ is weakly relatively compact in $L^2(Q_T)$ by (\ref{lem:entr:2:1bis}) in Lemma \ref{lem:entropy}, and, on the other hand, $ \sqrt{c_i^n}$ converges strongly in $L^2(Q_T)$, up to a subsequence, $J_i^n=\sqrt{c_i^n}\left[J_i^n/\sqrt{c_i^n}\right]$ converges weakly in $L^1(Q_T)$ for all $T\in (0,\infty)$.

{\bf To identify the limit $J_i$}, let us analyze the convergence of the individual terms in $J_i^n =-d_i^n\nabla c_i^n-\eta^n d_i^n\nabla(c_i^n)^p-d_i^nz_ic_i^n \nabla \Phi^n$.

First, $\sqrt{c_i^n}$ converges to $\sqrt{c_i}$ in $L^2(Q_T)$ and using $(\ref{sqrt(c)})$, up to a subsequence, $\nabla \sqrt{c_i^n}$ weakly converges in $L^2(0,T;L^2_{loc}(\Omega))$, so that $\nabla c_i^n=2\sqrt{c_i^n}\nabla\sqrt{c_i^n}$ weakly converges in $L^1(0,T;L^1_{loc}(\Omega))$ and the limit is necessarily $\nabla c_i$. Since $d_i^n$ converges pointwise almost everywhere and stays bounded, it follows that $d_i^n\nabla c_i^n\to d_i\nabla c_i$ weakly in $L^1(0,T;L^1_{loc}(\Omega))$ as well.

As a consequence of this, $J_i^n+d_i^n\nabla c_i^n$ is also weakly relatively compact in $L^1(0,T;L^1_{loc}(\Omega))$. Let us show that
\begin{equation}\label{JJ}
\lim_{n\to\infty} J_i^n+d_i^n\nabla c_i^n=-z_id_ic_i\nabla\Phi.
\end{equation}
Remark that
$$J_i^n+d_i^n\nabla c_i^n=d_i^n(\eta^n\nabla (c_i^n)^p+z_ic_i^n\nabla \Phi^n)=d_i^n\sqrt{c_i^n}\frac{\eta^n\nabla (c_i^n)^p+z_ic_i^n\nabla \Phi^n}{\sqrt{c_i^n}}.$$
From Lemma~\ref{lem:entr:2} $(i)$, we know that the quotient on the right-hand side is weakly relatively compact in $L^2(0,T;L^2_{loc}(\Omega))$. Let us show that it satisfies
\begin{equation}\label{Ji2}
\lim_{n\to\infty} \frac{\eta^n\nabla (c_i^n)^p+z_ic_i^n\nabla \Phi^n}{\sqrt{c_i^n}}=z_i\sqrt{c_i}\nabla\Phi.
\end{equation}
Since $\sqrt{c_i^n}\to \sqrt{c_i}$ strongly in $L^2(Q_T)$ and $d_i^n\to d_i$ a.e.\ with uniform bound, (\ref{JJ}) will then follow. It is actually sufficient to prove that (\ref{Ji2}) holds in the sense of distributions.

From the weak $L^2$-convergence $\nabla\Phi^n\to\nabla\Phi$ and the strong $L^2$-convergence $\sqrt{c_i^n}\to\sqrt{c_i}$, we have $\sqrt{c_i^n}\nabla\Phi^n\to\sqrt{c_i}\nabla\Phi$ weakly in $L^1(Q_T)$. For the remaining term, note that $\sqrt{c_i^n}\nabla (c_i^n)^p=\frac{p-\frac12}{p}\nabla(c_i^n)^{p-1/2}$ and recall that $\eta^n\int_\Omega (c_i^n)^p$ is bounded independently of $\eta^n$ from Lemma~\ref{lem:entropy}. For an arbitrary test function $\ph\in C_0^\infty(\Omega)$, we compute
\begin{align*}
\left|\int_\Omega \eta^n\pa_{x_k}(c_i^n)^{p-\frac12}\ph\right|&=\eta^n\left|\int_\Omega(c_i^n)^{p-\frac12}\pa_{x_k}\ph\right|\leq \eta^n\int_\Omega (c_i^n)^{p-\frac12}\|\nabla\ph\|_{L^\infty(\Omega)}\\
	&\leq C\eta^n\left(\int_\Omega (c_i^n)^p\right)^{\frac{p-\frac12}p}\|\nabla\ph\|_{L^\infty(\Omega)}\leq C(\eta^n)^{1/2p}\|\nabla\ph\|_{L^\infty}\overset{n\to\infty}\longrightarrow0.
\end{align*}
As a result, we see that $\eta^n \sqrt{c_i^n}\nabla (c_i^n)^p$ tends to zero in a distributional sense. Those arguments finally yield (\ref{JJ}) and $J_i=-d_i\nabla c_i-d_iz_ic_i \nabla \Phi$. Note in passing that we proved $c_i\in L^1(0,T;W^{1,1}_{loc}(\Omega))$.\\

\noindent{\bf Step 3: $(c_i,\Phi)$ is a solution.} We already proved that $\Phi$ is a solution of (\ref{eq:th1:2}). For $c^n$, we may pass to the limit in the approximate variational problem (\ref{eq:prop1:main}), namely: for all $T>0$ and all test-functions $\psi$ with $\psi(T)=0$
\begin{equation*}
 \int_{Q_{T_0}}-c_i^n\partial_t\psi -J_i^n\nabla \psi =\int_{\Omega} c_i^{0n} \psi(0)+\int_{Q_{T_0}}f_i^n(c^n)\psi.
\end{equation*}
Note that $f_i^n(c^n)$ converges pointwise a.e.\ to $f_i(c)$ and stays bounded. Together with the $L^1(Q_T)$-convergence of $c_i^n$ and the weak $L^1(Q_T)$-convergence of $J_i^n$, this implies that $(c_i,\Phi)$ is solution of (\ref{eq:th1:1}) in Theorem \ref{th1}.\\

The last point to prove Theorem \ref{th1} is\\
\noindent{\bf Step 4: $c_i\in C([0,\infty); L^1(\Omega))$.} Note that by $c_i(t)-c_i^0+\int_0^tdiv\,J_i=\int_0^tf_i(c)$, we see that $c_i$ has a continuous representation from $[0,\infty)$ into $W^{-1,1}(\Omega)$. Coupled with the estimate (\ref{clogc}), we may deduce that $t\to c(t)$ is also continuous for the weak topology of $L^1(\Omega)$. Actually, we will now prove that $t\to c_i(t)\in L^1(\Omega)$ is continuous. This is not so obvious. Actually, we will prove that, for each $k\geq 2$
\begin{equation}\label{contTk}
t\in [0,\infty) \to T_k(c_i(t))\in L^1(\Omega) \;{\rm has \;right- \;and \;left-limits\; at \;each\;} t_0\in [0,\infty),
\end{equation}
where $T_k:[0,\infty)\to [0,\infty)$ is a concave $C^2$-function such that $0\leq T_k'\leq 1$ and
$$\forall r\in [0,k],\;T_k(r)=r,\; \forall r\in (k+1,\infty),\;T_k(r)=k+1/2.$$
Thus, the continuity of $t\to c_i(t)\in L^1(\Omega)$ will follow. Indeed, let us write
$$\int_\Omega|c_i(t)-c_i(s)|\leq\int_\Omega |c_i(t)-T_k(c_i(t))|+|T_k(c_i(t)-T_k(c_i(s))|+|T_k(c_i(s))-c_i(s)|.$$

By the estimate (\ref{clogc}) and the definition of $T_k$
$$\int_\Omega|c_i(t)-T_k(c_i(t)|\leq \int_{[c_i(t)\geq k]}c_i(t)\leq \frac{1}{\log k}\int_{[c_i(t)\geq k]}c_i(t)\log c_i(t)\leq \frac{M}{\log k},$$
and the same with $t$ replaced by $s$. Thus, assuming (\ref{contTk}), we have for all $t_0\in [0,\infty)$
$$\limsup_{t,s\to t_0, t>s>t_0}\int_\Omega |c_i(t)-c_i(s)|\leq \frac{2M}{\log k}.$$
Letting $k\to\infty$, this proves that $t\to c_i(t)$ has a right-limit at each $t_0$ and similarly for the left limits. Therefore, the continuity of $t\to c(t)$ holds with values in $L^1(\Omega)$ and not only with values in $W^{-1,1}(\Omega)$.\\

The proof of (\ref{contTk}) will be done in several steps (in the same spirit as in \cite{Rolland1}).  Let $S\in C^2([0,\infty))$ such that $S'(r)=0$ for $r$ large. From the equation in $c_i^n$, we have
\begin{equation}\label{S}
\partial_tS(c_i^n)+div\,\left(S'(c_i^n)J_i^n\right)=S''(c_i^n)\nabla c_i^nJ_i^n+S'(c_i^n)f_i^n(c^n).
\end{equation}
Recall that $J_i^n/\sqrt{c_i^n}$ is bounded in $L^2(Q_T)$ so that, for all $k>0$,
\begin{equation}\label{Ji}
\int_{[c_i^n\leq k]}(J_i^n)^2\leq k\int_\Omega \frac{(J_i^n)^2}{c_i^n}\leq C\,k.
\end{equation}
Since $\nabla \Phi^n$ is bounded in $L^2(Q_T)$, we deduce that
\begin{equation}\label{nablaci}
\int_{[c_i^n\leq k]}|\nabla c_i^n|^2\leq \int_{[c_i^n\leq k]}[|\nabla c_i^n|\left(1+\eta p(c_i^n)^{p-1}\right)]^2= \int_{[c_i^n\leq k]}\left|\frac{J_i^n}{d_i^n}-z_ic_i^n\nabla\Phi^n\right|^2\leq C(k).
\end{equation}
It follows from these estimates that the right-hand side $\mu_n:=S''(c_i^n)\nabla c_i^nJ_i^n+S'(c_i^n)f_i^n(c^n)$ is bounded in $L^1(Q_T)$. Up to a subsequence, it converges to a finite measure  $\mu$ (depending on $S$). We will remember that, from passing to the limit in (\ref{S}), we may write for a.e.\ $0<s<t<T$ and for all $\psi\in C^\infty(\overline{\Omega})$:
\begin{equation}\label{SS}
\left\{
\begin{array}{l}
\int_\Omega \psi[S(c_i(t))-S(c_i(s))]+\nabla\psi\int_s^tW(\sigma)d\sigma=\int_{(s,t)\times\Omega}\psi\,d\mu,\\
{\rm where\;} W\in L^2(Q_T)^N.
\end{array}
\right.
\end{equation}
The fact that $W= \lim_{n\to\infty}S'(c_i^n)J_i^n\in L^2(Q_T)^N$ follows from (\ref{Ji}) and the choice of $S$. Note also that it follows from (\ref{nablaci}) that $S(c_i^n)$ is bounded in $L^2(0,T; W^{1,2}(\Omega))$. Let us choose $S=T_k$ in (\ref{SS}). Letting $s,t$ decrease to $t_0\in [0,T)$, we deduce that the limit as $t$ decreases to $t_0$ of $t\to T_k(c_i(t))$ exists in the sense of distributions. Since $T_k(c_i(t))$ is bounded in $L^\infty(\Omega)$, this limit also holds weakly in $L^2(\Omega)$ (at least). Let us denote it by $v(t_0)$. To prove that the convergence holds strongly in $L^2(\Omega)$, it is sufficient to prove that
\begin{equation}\label{strongg}
\limsup_{t\to t_0^+}\int_\Omega T_k(c_i(t))^2\leq\int_\Omega v(t_0)^2,
\end{equation}
and this will end the proof of {\bf Step 4} (up to the same analysis for the left limits).

Note first that, by letting $s$ decrease to $t_0$ in (\ref{SS}), we have for a.e. $t>t_0$
$$\int_\Omega \psi[T_k(c_i(t))-v(t_0)]+\nabla\psi \cdot \int_{t_0}^t\nabla W(\sigma)d\sigma=\int_{]t_0,t)\times\Omega}\psi\,d\mu.$$
Applying this to a sequence $\psi=\psi_p$ converging to $w(t):=T_k(c_i(t))$, we obtain
$$\int_\Omega w(t)^2-w(t)v(t_0)+\nabla w(t)\cdot \int_{t_0}^t\nabla W(\sigma)d\sigma\leq \|w(t)\psi\|_{L^\infty}\int_{]t_0,t]\times\Omega}\,d|\mu|.$$
We would like to pass to the limit as $t\to t_0^+$ in this estimate, but it is not obvious how to control the gradient term. Thus we integrate this equation in $t$ from $t_0$ to $t_0+h$ to obtain
$$\frac{1}{h}\int_{t_0}^{t_0+h}\int_\Omega w(t)^2dt-\int_\Omega v(t_0)\frac{1}{h}\int_{t_0}^{t_0+h}w(t)dt+\int_\Omega\frac{1}{h}\int_{t_0}^{t_0+h}\!\!dt\left\{\nabla w(t)\! \cdot \! \int_{t_0}^t\nabla W(\sigma)d\sigma\right\}\leq C\int_{]t_0,t_0+h]\times\Omega}\,d|\mu|.$$
Since $\frac{1}{h}\int_{t_0}^{t_0+h}w(t)dt$ converges weakly in $L^2(\Omega)$ to $v(t_0)$, if we are able to show that the gradient term tends to zero, we will deduce
\begin{equation}\label{average}
\limsup_{t\to t_0^+}\frac{1}{h}\int_{t_0}^{t_0+h}\int_\Omega w(t)^2dt\leq \int_\Omega v(t_0)^2.
\end{equation}
To estimate the gradient term, we use
$$\left|\int_{t_0}^t\nabla W(\sigma)d\sigma\right|\leq h^{1/2}\left[\int_{t_0}^{t_0+h}|\nabla W|^2(\sigma)d\sigma\right]^{1/2}\!\!\!,\;\; \left|\int_{t_0}^t\nabla w(t)dt\right|\leq h^{1/2}\left[\int_{t_0}^{t_0+h}|\nabla w|^2(t)dt\right]^{1/2}\!\!\!,$$
so that
$$ \left|\int_\Omega\frac{1}{h}\int_{t_0}^{t_0+h}\!\!dt\left\{\nabla w(t)\!\cdot \! \int_{t_0}^t\nabla W(\sigma)d\sigma\right\}\right|\leq \left[\int_{(t_0, t_0+h)\times\Omega}|\nabla W|^2(\sigma)d\sigma\right]^{1/2}\left[\int_{(t_0,t_0+h)\times\Omega}|\nabla w|^2(t)dt\right]^{1/2}\!\!\!.$$
This proves, as expected, that the gradient term tends to zero as $h\to 0$, whence (\ref{average}).

Next, to deduce (\ref{strongg}) from (\ref{average}), it is sufficient to know that $\lim_{t\to t_0^+}\int_\Omega w(t)^2$ exists. This may be checked by applying (\ref{SS}) with $S:=(T_k)^2$ and $\psi\equiv 1$. We then obtain for a.e.\ $0<s<t<T$
$$\left|\int_\Omega w(t)^2-w(s)^2\right|\leq\int_{(s,t)\times\Omega}d|\mu|,$$
and the right-hand side tends to zero as $t,s$ decrease to $t_0$.

\qed

\section{Proof of Theorem~\ref{th2}.}\label{S5}

The proof is the same as for Theorem~\ref{th1}, except that, having Lemma~\ref{lem:entr:2} $(ii)$ at hand, we are able to obtain compactness results in better spaces and up to the boundary of $\Omega$.

We use the same approximation as in Theorem~\ref{th1}, except that we moreover ask that $\xi^n\to\xi$ in $L^q(\pa\Omega)$. We have at least convergence in the same function spaces for the approximate solutions $(c^n,\Phi^n)$. But now, $\nabla \sqrt{c_i^n}$ is bounded in $L^2(Q_T)$ (that is up the boundary of $\Omega$). By Sobolev embedding, it follows that $\sqrt{c_i^n}$ is bounded in $L^2(0,T; L^6(\Omega))$ and $c_i^n$ is therefore bounded in $L^1(0,T;L^3(\Omega))$. Moreover, we may write
\[\|\nabla c_i^n\|_{L^1(0,T;L^{3/2}(\Omega))}\leq\|(c_i^n)^{-1/2} \nabla c_i^n\|_{L^2(Q_T)}\|\sqrt{c_i^n}\|_{L^2(0,T;L^6(\Omega)  )}\leq C,\]
so that $c_i\in L^1(0,T;W^{1,\frac32}(\Omega))$; here we use the fact that, since $\nabla c_i^n$ converges weakly in $L^1(Q_T)$ to $\nabla c_i$, a barycentric sequence converges strongly in $L^1(Q_T)$ and a.e.\ to the same $\nabla c_i$ and this sequence remains bounded in $L^1(0,T; L^{\frac32}(\Omega))$ as well. Whence $\nabla c_i\in L^1(0,T;L^{\frac32}(\Omega))$ by Fatou's Lemma.\\

With respect to $\Phi^n$ and $\Phi$, Lemma~\ref{lem:entr:2} says that $\Delta \Phi^n$ is bounded in $L^2(Q_T)$. We deduce that $\Phi\in L^2(0,T;W^{2,2}(\Omega))$. Since $c_i\in C([0,\infty);L^1(\Omega))$, it follows from the equation in $\Phi$ that $\Phi\in C([0,\infty);L^1(\Omega))$ (at least!).
But $\Phi$ being bounded in $L^\infty(0,T;W^{1,2}(\Omega))$, which is embedded into $L^\infty(0,T; L^6(\Omega))$, it follows that $\Phi\in C([0,\infty); L^r(\Omega))$ for $r\in [1,6)$.\\
\qed

\section{Appendix}
\begin{lemma} \label{supbound} Let $w$ be a regular nonnegative solution to
$$\partial_tw-div\left(d\,\nabla h(w)+ d\,w\,A\right)=g\;on\;Q_T,\;\partial_\nu h(w)+wA\cdot\nu=0\;on\;\Sigma_T,$$
where $d,g\in L^\infty(Q_T)$ with $0<d_m\leq d\leq d_M<\infty$,
$A\in L^\infty(Q_T, \R^N)$ and $h(w)=w+\eta T_M(w)$ where $T_M$ is a regular nondecreasing function. Then: $\|w\|_{L^\infty(Q_T)}\leq K$ where $K$ depends only on $\|g\|_{L^\infty(Q_T)}, \|A\|_{L^\infty(Q_T)}, \|w(0)\|_{L^\infty(\Omega)}, d_m, d_M, T$.
\end{lemma}

\begin{proof} First, we write the equation satisfied by $W:=e^{-\omega t} w$ where $\omega>0$ will be chosen large enough later:
$$\partial_t W-div\left(d\,e^{-\omega t}\nabla h(e^{\omega t}W)+ d\,W\,A\right)+\omega W=e^{-\omega t}g\;on\;Q_T,\;e^{-\omega t}\partial_\nu h(e^{\omega t}W)+W\,A\cdot\nu=0\;on\;\Sigma_T.$$
Let $k\in [0,\infty)$. We set $W_k:=(W-k)^+$ and we multiply the equation by $W_k$ to obtain
$$\partial_t\frac{1}{2}\int_\Omega W_k^2+\int_\Omega d|\nabla W_k|^2\leq \int_\Omega -d\,W\nabla W_k\cdot A-\omega W_kW+\|g\|_{L^\infty}W_k.$$
(Here we used $\int_\Omega\nabla W_k\nabla W^p\geq 0$ and $h'\geq 1$). By Young's inequality
$$-\int_\Omega dW\nabla W_k\cdot A\leq \int_\Omega \frac{d}{2}|\nabla W_k|^2+\frac{d_M}{2}\|A\|_{L^\infty}^2\int_{[W(t)\geq k]}W^2,$$
$$\int_{[W(t)\geq k]}\omega\,k\,W+\|g\|_{L^\infty}W_k\leq\int_{[W(t)\geq k]} \frac{\omega}{2}(W^2+k^2)+\frac{\|g\|_{L^\infty}}{2}(W^2+1).$$
We now choose $\omega\geq d_M\|A\|_{L^\infty}^2+\|g\|_{L^\infty}^2$ to get
\begin{equation}\label{allk}
\partial_t\int_\Omega W_k^2+{d_m}\int_\Omega |\nabla W_k|^2\leq \left[\omega k^2+\|g\|_{L^\infty}\right]\int_{[W(t)\geq k]}dx,
\end{equation}
where $[W(t)\geq k]=\{x\in\Omega; W(t)\geq k\}$. Choosing $k\geq \|W(0)\|_{L^\infty(\Omega)}+1=\|w(0)\|_{L^\infty(\Omega)}+1$, we obtain after integration in time
$$\sup_{t\in [0,T]}\int_\Omega W_k^2(t)+d_m\int_{Q_T}|\nabla W_k|^2\leq (\omega+\|g\|_{L^\infty})k^2\int_{[W\geq k]}dx\,dt,$$
where $[W\geq k]=\{(t,x)\in Q_T; W(t,x)\geq k\}$.
Now, we apply Theorem II.6.1 (with $r=q=2(N+2)N, \kappa=2/N$) together with Remark  II.6.1 in pages 102-103 of \cite{LSU} to obtain the conclusion of Lemma \ref{supbound}.
\end{proof}

\begin{lemma}\label{H-1} Under the assumptions of Lemma \ref{supbound}, we also have
$$\|\nabla w\|_{L^2(Q_T)}+\|\partial_t w\|_{L^2(0,T;W^{-1,2}(\Omega))}\leq K,$$
where $K$ depends on the same quantities as in Lemma \ref{supbound} and also on $p$.
\end{lemma}

\begin{proof} Writing (\ref{allk}) with $k=0$ gives the estimate on $\|\nabla w\|_{L^2(Q_T)}$ after integration in time. Next we see that
$$\|d\nabla h(w)+d\,w\,A\|_{L^2(Q_T)}\leq d_M\left\{[1+p\|w\|_{L^\infty}^{p-1}]\|\nabla w\|_{L^2(Q_T)}+|Q_T|\|w\|_{L^\infty}\|A\|_{L^\infty}\right\},$$
providing an $L^2(Q_T)$-estimate on $d\nabla h(w)+d\,w\,A$ in terms of the announced quantities. Whence the estimate of $\partial_t w$ in $L^2(0,T;W^{-1,2}(\Omega))$.
\end{proof}

\section*{Acknowledgement}
The second author was partly supported by the Center of Smart Interfaces, TU Darmstadt and Rennes M\'etropole.

\bibliography{literature}{}
\bibliographystyle{plain}
\end{document}